\documentclass[amstex,12pt]{article}
\usepackage{amsfonts}
\usepackage{mathrsfs}
\usepackage{amssymb,amsmath,amsthm,graphicx}
\newtheorem{theorem}{Theorem}[section]

\newtheorem{lemma}{Lemma}[section]
\newtheorem{definition}{Definition}[section]

\newtheorem{example}{Example}[section]

\newcommand{\beq}{\begin{equation}}
\newcommand{\eeq}{\end{equation}}
\newcommand{\beqn}{\begin{eqnarray}}
\newcommand{\eeqn}{\end{eqnarray}}

\allowdisplaybreaks \baselineskip 20pt
\begin{document}

\title{Pseudo almost periodic solutions for neutral type high-order Hopfield neural networks with
mixed time-varying delays and leakage delays on time scales\thanks{This work is supported by
the National Natural Sciences Foundation of People's Republic of
China under Grants 11361072 and 11461082.} }
\author{ Yongkun Li$^a$\thanks{%
The corresponding author.}, Xiaofang Meng$^a$ and Lianglin Xiong$^b$  \\
$^a$Department of Mathematics, Yunnan University\\
Kunming, Yunnan 650091\\
People's Republic of China\\
$^b$School of Mathematics and Computer Science\\
 Yunnan Minzu University, Kunming, Yunnan 650500\\
 People's Republic of   China}

\date{}
\maketitle \allowdisplaybreaks
\begin{abstract}
In this paper, a class of neutral type high-order Hopfield neural networks with mixed time-varying delays and leakage delays on time scales is proposed. Based on the exponential dichotomy of linear dynamic equations on time scales, Banach's fixed point theorem and the
theory of calculus on time scales, some sufficient conditions are obtained for the existence and  global exponential stability of pseudo almost periodic solutions for this class of neural networks. Our results are completely new. Finally, we present an example to illustrate our results are effective.  Our example also shows that the continuous-time neural network and its discrete-time analogue have the same dynamical behaviors for the pseudo almost periodicity.
\end{abstract}

\allowdisplaybreaks
\section{Introduction}

\setcounter{section}{1}
\setcounter{equation}{0}
\indent

Due to the fact that high-order Hopfield neural
networks (HHNNs) have stronger approximation property, faster convergence
rate, greater storage capacity, and higher fault tolerance than lower-order neural networks, high-order Hopfield neural networks have
been the object of intensive analysis by numerous authors in recent
years. In particular, there have been extensive results on the problem
of the existence and stability of equilibrium points, periodic solutions and
almost periodic solutions  of  HHNNs in the literature. We refer the reader to [1-9] and the references cited therein. For example, authors of \cite{4}  studied the problem of global exponential stability properties of such high-order Hopfield-type neural networks by utilizing Lyapunov functions; authors of \cite{5} derived some sufficient conditions for the global asymptotic stability of equilibrium point of HHNNs with constant time delays in terms of linear matrix inequality.

In fact, it is natural and important that systems will contain some information about the derivative of the past state to further describe and model the dynamics for such complex neural reactions \cite{15}, many authors investigated the dynamical behaviors of neutral type neural networks. For example, authors of [11-15] studied the stability, periodic solutions and almost periodic solutions for different classes of neutral type neural networks, respectively.

As we known, time delays inevitably exist in biological and artificial neural networks because of the finite switching speed of neurons and amplifiers \cite{dd11}, which can also affect the stability of neural network systems and may lead to some complex dynamic behaviors such as oscillation, chaos and instability. In reality, the mixed time-varying delays should be taken into account when modeling realistic neural networks \cite{17}. A leakage delay, which is the time delay in the leakage term of the systems and a considerable factor affecting dynamics for the worse in the systems, is being put to use in the problem of stability for neural networks. Such time delays in the leakage term are difficult to handle but has great impact on the dynamical behavior of neural networks [18-21]. Therefore, it is significant to consider neural networks with time delays in leakage terms.

Besides,
both continuous-time and discrete-time   neural networks have equally importance
in various applications and the theory of time scales was initiated by  Hilger \cite{d7} in his Ph.D. thesis in 1988, which can unify the continuous and discrete cases. Many authors have studied the dynamics of neural networks on time scales [15,23-25]. For example, in paper \cite{d11}, some sufficient conditions for the existence and global exponential stability of almost automorphic solutions for a class of neutral type HHNNs with delays in leakage terms on time scales are obtained; in paper \cite{d15}, the existence and global exponential stability of anti-periodic solutions for competitive neural networks with delays in the leakage terms on time scales are investigated.

On the other hand, the concept of pseudo-almost periodicity, which is the central subject in this paper, was introduced by
Zhang \cite{42} in the early nineties. As pointed out by Dads et al. in \cite{24}, it would be of great interest to study the dynamics of pseudo almost periodic systems with time delays. Pseudo almost periodic solutions in the context of differential equations were studied by several authors in [28-41]. For example, authors of \cite{32} studied the existence and the global exponential stability of the positive pseudo almost periodic solutions, which are more general and complicated than periodic and almost periodic solutions; authors of \cite{37} studied the existence and uniqueness of pseudo almost periodic solution of the shunting inhibitory cellular neural networks
 with time-varying delays in the leakage terms by using the exponential dichotomy theory and contraction mapping fixed point theorem. Recently, the concept of the pseudo almost periodic function on time scales has been introduced by Li and Wang \cite{39}. However, to the best of our knowledge, there is no paper published on the existence and stability of pseudo almost periodic solutions for neutral type HHNNs with mixed time-varying delays and leakage delays on time scales, which is very important in theories and applications and also is a very challenging problem.

Motivated by the above discussion, in this paper, we propose a neutral type high-order Hopfield neural network with
 mixed time-varying delays and leakage delays on time scales as follows:
\begin{eqnarray}\label{a1}
x_{i}^{\nabla}(t)&=&-c_{i}(t)x_{i}(t-\delta_{i}(t))+\sum_{j=1}^{n}a_{ij}(t)f_{j}(x_{j}(t))\nonumber\\
&&+\sum_{j=1}^{n}b_{ij}(t)g_{j}(x_{j}(t-\tau_{ij}(t)))
+\sum_{j=1}^{n}d_{ij}(t)\int_{t-\sigma_{ij}(t)}^{t}h_{j}(x_{j}^{\nabla}(s))\nabla s\nonumber\\
&&+\sum_{j=1}^{n}\sum_{l=1}^{n}T_{ijl}(t)k_{j}(x_{j}(t-\xi_{ijl}(t)))k_{l}(x_{l}(t-\zeta_{ijl}(t)))
+I_{i}(t),\,\,t\in \mathbb{T},
\end{eqnarray}
where $\mathbb{T}$ is an almost periodic time scale, $i=1,2,\ldots,n$, $n$ corresponds to the number of units in a neural network; $x_{i}(t)$ denotes the activation of the $i$th neuron at time $t$; $c_{i}(t)>0$ represents the rate with the $i$th unit will reset its potential to the resting state in isolation when disconnected from the
network and external inputs at time $t$; $a_{ij}(t)$, $b_{ij}(t)$ and $d_{ij}(t)$ represent the delayed
strengths of connectivity and neutral delayed strengths of connectivity between cell $i$ and $j$ at time $t$, respectively; $T_{ijl}(t)$ denotes the second-order connection weights of the neural network; $f_{j}$, $g_{j}$, $h_{j}$ and $k_{j}$ are the activation functions in system \eqref{a1}; $I_{i}(t)$ is an
external input on the $i$th unit at time $t$; $\delta_{i}$ denotes the leakage delay at time $t$ and satisfy $t-\delta_{i}(t)\in \mathbb{T}$, $\tau_{ij}$, $\sigma_{ij}$, $\xi_{ijl}$ and $\zeta_{ijl}$ are transmission delays at time $t$ and satisfy $t-\tau_{ij}(t)\in \mathbb{T}$ , $t-\sigma_{ij}(t)\in \mathbb{T}$, $t-\xi_{ijl}(t)\in \mathbb{T}$, $t-\zeta_{ijl}(t)\in \mathbb{T}$ for $t\in \mathbb{T}$.

For convenience, we let $[a,b]_{\mathbb{T}}=\{t|t\in[a,b]\cap\mathbb{T}\}$. And we introduce the following notations:
\[c_{i}^{+}=\sup\limits_{t\in\mathbb{T}}|c_{i}(t)|,\,\,\,c_{i}^{-}=\inf\limits_{t\in\mathbb{T}}|c_{i}(t)|,\,\,\,
\delta_{i}^{+}=\sup\limits_{t\in\mathbb{T}}|\delta_{i}(t)|,\,\,\,\tau_{ij}^{+}=\sup\limits_{t\in\mathbb{T}}|\tau_{ij}(t)|,\]
\[\sigma_{ij}^{+}=\sup\limits_{t\in\mathbb{T}}|\sigma_{ij}(t)|,\,\,\,
\xi_{ijl}^{+}=\sup\limits_{t\in\mathbb{T}}|\xi_{ijl}(t)|,\,\,\,
\zeta_{ijl}^{+}=\sup\limits_{t\in\mathbb{T}}|\zeta_{ijl}(t)|,\,\,\,a_{ij}^{+}=\sup\limits_{t\in\mathbb{T}}|a_{ij}(t)|,
\]
\[
b_{ij}^{+}=\sup\limits_{t\in\mathbb{T}}|b_{ij}(t)|,\,\,\,d_{ij}^{+}=\sup\limits_{t\in\mathbb{T}}|d_{ij}(t)|,\,\,\,
T_{ijl}^{+}=\sup\limits_{t\in\mathbb{T}}|T_{ijl}(t)|,
\,\,i,j,l=1,2,\dots,n.\]

The initial condition associated with system \eqref{a1} is of the form
\begin{eqnarray}\label{a2}
x_{i}(s)=\varphi_{i}(s),\,\,
x_{i}^{\nabla}(s)=\varphi_{i}^{\nabla}(s),\,\,\,s\in[-\theta,0]_{\mathbb{T}},
\end{eqnarray}
where $\theta=\max\{\delta,\tau,\sigma,\xi,\zeta\}$, $\delta=\max\limits_{1\leq i\leq n}\{\delta_{i}^{+}\}$, $\tau=\max\limits_{1\leq i,j\leq n}\{\tau_{ij}^{+}\}$, $\sigma=\max\limits_{1\leq i,j\leq n}\{\sigma_{ij}^{+}\}$,
$\xi=\max\limits_{1\leq i,j,l\leq n}\{\xi_{ijl}^{+}\}$,
$\zeta=\max\limits_{1\leq i,j,l\leq n}\{\zeta_{ijl}^{+}\}$, $i,j,l=1,2,\ldots,n$, $\varphi_{k}(\cdot)$ denotes a real-valued bounded $\nabla$-differentiable
function defined on $[-\theta,0]_{\mathbb{T}}$.

This paper is organized as follows. In Section 2, we introduce some definitions, make some preparations for later sections and   extend the almost periodic theory on time scales with the delta derivative to that with the nabla derivative. In Section 3, by utilizing
Banach's fixed point theorem and the theory of calculus on time scales, we present
some sufficient conditions for the existence of pseudo almost periodic solutions of $\eqref{a1}$.
In Section 4, we prove that the pseudo almost periodic solution obtained in Section 3 is globally
exponentially stable. In Section 5, we give an example to demonstrate the feasibility of our
results.

\section{Preliminaries}

\setcounter{section}{2}
\setcounter{equation}{0}
\indent

In this section, we shall first recall some fundamental definitions and lemmas. Also, we extend the almost periodic theory on time scales with the delta derivative to that with the nabla derivative.

A time scale $\mathbb{T}$
is an arbitrary nonempty closed subset of the real set $\mathbb{R}$ with the topology and ordering inherited from $\mathbb{R}$.
The forward jump operator
$\sigma:\mathbb{T}\rightarrow\mathbb{T}$ is defined by $\sigma(t)=\inf\big\{s\in \mathbb{T},s>t\big\}$ for all $t\in \mathbb{T}$,
while the backward jump operator $\rho:\mathbb{T}\rightarrow\mathbb{T}$ is defined by
$\rho(t)=\sup\big\{s\in \mathbb{T},s<t\big\}$ for all $t\in\mathbb{T}$.

A point $t\in\mathbb{T}$ is called left-dense if $t>\inf\mathbb{T}$
and $\rho(t)=t$, left-scattered if $\rho(t)<t$, right-dense if
$t<\sup\mathbb{T}$ and $\sigma(t)=t$, and right-scattered if
$\sigma(t)>t$. If $\mathbb{T}$ has a left-scattered maximum $m$,
then $\mathbb{T}^k=\mathbb{T}\setminus\{m\}$; otherwise
$\mathbb{T}^k=\mathbb{T}$. If $\mathbb{T}$ has a right-scattered
minimum $m$, then $\mathbb{T}_k=\mathbb{T}\setminus\{m\}$; otherwise
$\mathbb{T}_k=\mathbb{T}$. Finally, the backwards graininess function
$\nu: \mathbb{T}_{k}\rightarrow [0,\infty)$ is defined by $\nu(t)=t-\rho(t)$.

A function $f:\mathbb{T}\rightarrow\mathbb{R}$ is ld-continuous provided
it is continuous at left-dense point in $\mathbb{T}$ and its right-side
limits exist at right-dense points in $\mathbb{T}$.

\begin{definition}$(\cite{d8,d9})$
Let $f:\mathbb{T} \rightarrow \mathbb{R}$ be a function and
$t\in \mathbb{T}_{k}$. Then we define $f^{\nabla}(t)$ to be the
number (provided its exists) with the property that given any
$\varepsilon>0$, there is a neighborhood $U$ of
$t$ (i.e, $U=(t-\delta,t+\delta)\cap \mathbb{T}$ for some
$\delta>0$) such that
\[
|f(\rho(t))-f(s)-f^{\nabla}(t)(\rho(t)-s)|\leq\varepsilon|\rho(t)-s|
\]
for all $s\in U$, we call $f^{\nabla}(t)$ the nabla derivative of $f$ at $t$.
\end{definition}

Let $f:\mathbb{T} \rightarrow \mathbb{R}$ be ld-continuous. If $F^{\nabla}(t)=f(t)$, then we
define the nabla integral by $\int_a^{b}f(t)\nabla t=F(b)-F(a)$.

A function $p:\mathbb{T}\rightarrow\mathbb{R}$ is called $\nu$-regressive
if $1-\nu(t)p(t)\neq 0$ for all $t\in \mathbb{T}_k$. The set of all $\nu$-regressive and
left-dense continuous functions $p:\mathbb{T}\rightarrow\mathbb{R}$ will
be denoted by $\mathcal{R}_{\nu}=\mathcal{R}_{\nu}(\mathbb{T})=\mathcal{R}_{\nu}(\mathbb{T},\mathbb{R})$. We
define the set $\mathcal{R}_{\nu}^{+}=\mathcal{R}_{\nu}^{+}(\mathbb{T},\mathbb{R})=\{p\in \mathcal{R}_{\nu}:1-\nu(t)p(t)>0,\,\,\forall
t\in\mathbb{T}\}$.

If $p\in \mathcal {R}_{\nu}$, then we define the nabla exponential function by
\begin{equation*}
    \hat{e}_{p}(t,s)=\exp\bigg(\int_{s}^{t}\hat{\xi}_{\nu(\tau)}\big(p(\tau)\big)\nabla\tau\bigg),\,\,\mathrm{for}\,\,t,s\in\mathbb{T}
\end{equation*}
with the $\nu$-cylinder transformation
\[
\hat{\xi}_h(z)=\bigg\{\begin{array}{ll} -\frac{\mathrm{log}(1-hz)}{h} &{\rm
if}\,h\neq
0,\\
z &{\rm if}\,h=0.\\
\end{array}
\]
Let $p,q\in\mathcal{R}_{\nu}$, then we define a circle plus addition by
$(p \oplus_{\nu} q)(t)= p(t)+q(t)-\nu(t) p(t)q(t)$, for all $t\in\mathbb{T}_{k}$. For $p\in\mathcal{R}_{\nu}$, define a circle minus
$p$ by $\ominus_{\nu} p=-\frac{p}{1-\nu p}$.

\begin{lemma}$(\cite{d8,d9})$
Let $p,q\in\mathcal{R}_{\nu}$, and $s,t,r\in\mathbb{T}$. Then
\begin{itemize}
\item[$(i)$] $\hat{e}_{0}(t,s)\equiv 1$ and $\hat{e}_{p}(t,t)\equiv 1$;
\item[$(ii)$] $\hat{e}_{p}(\rho(t),s)=(1-\nu(t)p(t))\hat{e}_{p}(t,s)$;
\item[$(iii)$] $\hat{e}_{p}(t,s)=\frac{1}{\hat{e}_{p}(s,t)}=\hat{e}_{\ominus_{\nu} p}(s,t)$;
\item[$(iv)$] $\hat{e}_{p}(t,s)\hat{e}_{p}(s,r)=\hat{e}_{p}(t,r)$;
\item[$(v)$] $(\hat{e}_{p}(t,s))^{\nabla}=p(t)\hat{e}_{p}(t,s)$.
\end{itemize}
\end{lemma}

\begin{lemma}$(\cite{d8,d9})$ Let $f,g$ be nabla differentiable functions on $\mathbb{T}$, then
\begin{itemize}
\item[$(i)$] $(v_{1}f+v_{2}g)^{\nabla}=v_{1}f^{\nabla}+v_{2}g^{\nabla}$, for any constants $v_{1},v_{2}$;
\item[$(ii)$] $(fg)^{\nabla}(t)=f^{\nabla}(t)g(t)+f(\rho(t))g^{\nabla}(t)=f(t)g^{\nabla}(t)+f^{\nabla}(t)g(\rho(t))$;
\item[$(iii)$] If $f$ and $f^{\nabla}$ are continuous, then $(\int_{a}^{t}f(t,s)\nabla s)^{\nabla}=f(\rho(t),t)+\int_{a}^{t}f(t,s)\nabla s.$
\end{itemize}
\end{lemma}

\begin{lemma}$(\cite{d8,d9})$ Assume $p\in\mathcal{R}_{\nu}$ and $t_{0}\in\mathbb{T}$. If $1-\nu(t)p(t)>0$ for $t\in\mathbb{T}$, then $\hat{e}_{p}(t,t_{0})>0$ for all $t\in\mathbb{T}$.
\end{lemma}

\begin{lemma}\label{lem24} Suppose that $f(t)$ is an ld-continuous function
and $c(t)$ is a positive ld-continuous function which satisfies that
$c(t)\in\mathcal{R}_{\nu}^{+}$. Let
\begin{eqnarray*}
g(t)=\int_{t_{0}}^{t}\hat{e}_{-c}(t,\rho(s))f(s)\nabla s,
\end{eqnarray*}
where $t_{0}\in\mathbb{T}$, then
\begin{eqnarray*}
g^{\nabla}(t)=f(t)-c(t)\int_{t_{0}}^{t}\hat{e}_{-c}(t,\rho(s))f(s)\nabla s.
\end{eqnarray*}
\end{lemma}
\begin{proof}
\begin{eqnarray*}
g^{\nabla}(t)&=&\bigg(\int_{t_{0}}^{t}\hat{e}_{-c}(t,\rho(s))f(s)\nabla s\bigg)^{\nabla}\\
&=&\bigg(\hat{e}_{-c}(t,t_{0})\int_{t_{0}}^{t}\hat{e}_{-c}(t_{0},\rho(s))f(s)\nabla s\bigg)^{\nabla}\\
&=&\hat{e}_{-c}(\rho(t),t_{0})\hat{e}_{-c}(t_{0},\rho(t))f(t)-c(t)\hat{e}_{-c}(t,t_{0})\int_{t_{0}}^{t}\hat{e}_{-c}(t_{0},\rho(s))f(s)\nabla s\\
&=&f(t)-c(t)\int_{t_{0}}^{t}\hat{e}_{-c}(t,\rho(s))f(s)\nabla s.
\end{eqnarray*}
\end{proof}
\begin{definition}$(\cite{43,d10})$
A time scale $\mathbb{T}$ is called an almost periodic time scale if
\begin{eqnarray*}
\Pi:=\{\tau\in \mathbb{R}:t\pm\tau\in \mathbb{T}, \forall t\in
\mathbb{T}\}\neq \{0\}.
\end{eqnarray*}
\end{definition}

\begin{definition}$(\cite{43,d10})$
Let $\mathbb{T}$ be an almost periodic time scale. A function $f \in
C(\mathbb{T},\mathbb{R}^{n})$ is called an almost periodic on $\mathbb{T}$,
if for any $\varepsilon >0$, the set
\begin{eqnarray*}
E(\varepsilon,f)=\{{\tau \in
\Pi:|f(t+\tau)-f(t)|<\varepsilon,\forall t\in \mathbb{T}}\}
\end{eqnarray*}
is a relatively dense in $\mathbb{T}$; that is, for any given $\varepsilon>0$,
there exists a constant $l(\varepsilon)>0$ such that each interval of
length $l(\varepsilon)$ contains at least one $\tau=\tau(\varepsilon)\in
E(\varepsilon,f)$ such that
\begin{eqnarray*}
|f(t+\tau)-f(t)|<\varepsilon,\,\, \forall t\in \mathbb{T}.
\end{eqnarray*}
The set $E(\varepsilon,f)$ is called the $\varepsilon$-translation set of $f(t)$,
$\tau$ is called the $\varepsilon$-translation number of $f(t)$ and
$l(\varepsilon)$ is called the contain interval length of $E(\varepsilon,f)$.
\end{definition}

Let $AP(\mathbb{T})=\{f\in C(\mathbb{T},\mathbb{R}^{n}): f$ is almost periodic\} and $BC(\mathbb{T},\mathbb{R}^{n})$ denote the
space of all bounded continuous functions from $\mathbb{T}$ to $\mathbb{R}^{n}$. Define the class of functions $PAP_{0}(\mathbb{T})$ as follows:
\begin{eqnarray*}
PAP_{0}(\mathbb{T})&=&\bigg\{f\in BC(\mathbb{T},\mathbb{R}^{n}):f\,\,is\,\,\nabla-measurable\,\,such\,\, that\\ &&\lim\limits_{r\rightarrow+\infty}\frac{1}{2r}\int_{t_{0}-r}^{t_{0}+r}|f(s)|\nabla s=0,\,\,where\,\,t_{0}\in\mathbb{T}, r\in\Lambda\bigg\}.
\end{eqnarray*}
Similar to Definition 4.1 in \cite{39}, we give
\begin{definition}\label{def25}  A function $f\in C(\mathbb{T},\mathbb{R}^{n})$ is called pseudo almost periodic if $f=g+\phi$,
where $g\in AP(\mathbb{T})$ and $\phi\in PAP_{0}(\mathbb{T})$. Denote by $PAP(\mathbb{T})$ the set of pseudo almost periodic functions.
\end{definition}
By Definition \ref{def25}, one can easily show that
\begin{lemma}\label{lem25}
 If $f,g\in PAP(\mathbb{T})$, then $f+g,fg\in PAP(\mathbb{T})$; if $f\in PAP(\mathbb{T}),\,
g\in AP(\mathbb{T})$, then $fg\in PAP(\mathbb{T})$.
\end{lemma}
\begin{lemma} If $f\in C(\mathbb{R},\mathbb{R})$ satisfies the Lipschitcz condition, $\varphi\in PAP(\mathbb{T})$ and $\theta\in C(\mathbb{T},\Lambda)$, then $f(\varphi(t-\theta(t)))\in PAP(\mathbb{T})$.
\end{lemma}
\begin{proof}
From Definition \ref{def25}, we have $\varphi=\varphi_{1}+\varphi_{2}$, where $\varphi_{1}\in AP(\mathbb{T})$ and $\varphi_{2}\in PAP_{0}(\mathbb{T})$. Set
\begin{eqnarray*}
E(t)&=&f(\varphi(t-\theta(t)))=f(\varphi_{1}(t-\theta(t)))+[f(\varphi_{1}(t-\theta(t))+\varphi_{2}(t-\theta(t)))-f(\varphi_{1}(t-\theta(t)))]\\
&:=&E_{1}(t)+E_{2}(t).
\end{eqnarray*}
Firstly, it follows from Theorem 2.11 in \cite{43} that  $E_{1}\in AP(\mathbb{T})$.  Then, we show that $E_{2}\in PAP_{0}(\mathbb{T})$ because
\begin{eqnarray*}
&&\lim\limits_{r\rightarrow+\infty}\frac{1}{2r}\int_{t_{0}-r}^{t_{0}+r}|E_{2}(s)|\nabla s\\
&=&\lim\limits_{r\rightarrow+\infty}\frac{1}{2r}\int_{t_{0}-r}^{t_{0}+r}|f(\varphi_{1}(s-\theta(s))+\varphi_{2}(s-\theta(s)))
-f(\varphi_{1}(s-\theta(s)))|\nabla s\\
&\leq&\lim\limits_{r\rightarrow+\infty}\frac{L}{2r}\int_{t_{0}-r}^{t_{0}+r}|\varphi_{2}(s-\theta(s))|\nabla s=0.
\end{eqnarray*}
Thus $E_{2}\in PAP_{0}(\mathbb{T})$. So $E\in PAP(\mathbb{T})$. The proof is complete.
\end{proof}

Similar to Definition 2.12 in \cite{43}, we give
\begin{definition}
Let $A(t)$ be an $n\times n$ matrix-valued function on $\mathbb{T}$. Then the linear system
\begin{eqnarray}\label{e21}
x^{\nabla}(t)=A(t)x(t),\,\, t\in\mathbb{T}
\end{eqnarray}
is said to admit an exponential dichotomy on $\mathbb{T}$ if there
exist positive constant $K, \alpha$, projection $P$ and the
fundamental solution matrix $X(t)$ of \eqref{e21}, satisfying
\begin{eqnarray*}
\|X(t)PX^{-1}(s)\|_{0} \leq K\hat{e}_{\ominus_{\nu}
\alpha}(t,s),\,\,
s, t \in\mathbb{T}, t \geq s,\\
\|X(t)(I-P)X^{-1}(s)\|_{0} \leq K\hat{e}_{\ominus_{\nu}
\alpha}(s,t),\,\, s, t \in\mathbb{T}, t \leq s,
\end{eqnarray*}
where $\|\cdot\|_{0}$ is a matrix norm on $\mathbb{T}$ $($say, for
example, if $A=(a_{ij})_{n\times m}$, then we can take
$\|A\|_{0}=(\sum\limits_{i=1}^{n}\sum\limits_{j=1}^{m}|a_{ij}|^{2})^{\frac{1}{2}})$.
\end{definition}

Consider the following almost periodic system
\begin{eqnarray}\label{e22}
x^{\nabla}(t)=A(t)x(t)+f(t),\,\, t \in \mathbb{T},
\end{eqnarray}
where $A(t)$ is an almost periodic matrix function, $f(t)$ is an
almost periodic vector function.
Similar to the proof of Theorem 5.2 in \cite{39}, we can get the following   lemma.
\begin{lemma}\label{lem27} Suppose that $A(t)$ is almost periodic, \eqref{e21} admits an exponential
dichotomy and the function $f\in PAP(\mathbb{T})$. Then \eqref{e22} has a unique bounded solution
$x\in PAP(\mathbb{T})$ that can be expressed as follows:
\[
x(t)=\int_{-\infty}^{t}X(t)PX^{-1}(\rho(s))f(s)\nabla
s-\int_t^{+\infty}X(t)(I-P)X^{-1}(\rho(s))f(s)\nabla s,
\]
where $X(t)$ is the fundamental solution matrix of \eqref{e21}.
\end{lemma}
Similar to the proof of Lemma 2.15 in \cite{43}, we have
\begin{lemma}\label{lem28}
Let $c_{i}(t):\mathbb{T}\rightarrow\mathbb{R}^{+}$ be an almost periodic function, $-c_{i}\in \mathcal{R}_{\nu}^{+}$,
$\mathbb{T}\in \Lambda$ and
\[m(c_{i})=\lim\limits_{T\rightarrow\infty}\frac{1}{T}\int_{t}^{t+T}c_{i}(s)\nabla s> 0,\,\,\,i=1,2,\ldots,n.\]
Then the linear system
\begin{eqnarray}\label{e23}
x^{\nabla}(t)=\mathrm{diag}(-c_{1}(t),-c_{2}(t),\dots,-c_{n}(t))x(t)
\end{eqnarray}
admits an exponential dichotomy on $\mathbb{T}$, where $m(c_{i})$ denote the mean-value of $c_{i}$, $i=1,2,\ldots,n$.
\end{lemma}

\section{The existence of pseudo almost periodic solution}

\setcounter{equation}{0}

\indent

In this section, we will state and prove the sufficient conditions for the existence of
pseudo almost periodic solutions of \eqref{a1}.

Let
\begin{eqnarray*}
\mathbb{B}=\big\{\varphi(t)=(\varphi_{1}(t),\varphi_{2}(t),\ldots,\varphi_{n}(t))^{T}:
\varphi_{i}(t),\varphi^{\nabla}_{i}(t)\in PAP(\mathbb{T}),\,\,i=1,2,\ldots,n\big\}
\end{eqnarray*}
 with the norm
$\|\varphi\|_{\mathbb{B}}=\sup\limits_{t\in\mathbb{T}}\|\varphi(t)\|$, where $\|\varphi(t)\|=\max\limits_{1\leq i\leq n}\{|\varphi_{i}(t)|,|\varphi^{\nabla}_{i}(t)|\}$, then $\mathbb{B}$ is a Banach space.

Throughout this paper, we assume that the following conditions hold:
\begin{itemize}
\item[$(H_{1})$] $c_{i}\in C(\mathbb{T},\mathbb{R}^{+})$ with $-c_{i}\in\mathcal{R}_{\nu}^{+}$, where $\mathcal {R}_{\nu}^{+}$ denotes the set of positively regressive functions from $\mathbb{T}$ to $\mathbb{R}$, $i=1,2,\dots,n$;
\item[$(H_{2})$] $a_{ij},b_{ij},d_{ij},T_{ijl}\in AP(\mathbb{T})$, $\delta_{i}, \tau_{ij}, \sigma_{ij}, \xi_{ijl}, \zeta_{ijl}\in C(\mathbb{T},\Lambda)$ and $I_{i}\in PAP(\mathbb{T})$;
\item[$(H_{3})$] Functions $f_{j},g_{j},h_{j},k_{j}\in C(\mathbb{R},\mathbb{R})$ and there exist positive constants $L_{j}^{f},L_{j}^{g},L_{j}^{h},L_{j}^{k}$ such that
\[
|f_{j}(u)-f_{j}(v)|\leq L_{j}^{f}|u-v|, |g_{j}(u)-g_{j}(v)|\leq L_{j}^{g}|u-v|,\]
\[
|h_{j}(u)-h_{j}(v)|\leq L_{j}^{h}|u-v|, |k_{j}(u)-k_{j}(v)|\leq L_{j}^{k}|u-v|,\]
where $u,v\in \mathbb{R}$ and $f_{j}(0)=g_{j}(0)=h_{j}(0)=k_{j}(0)=0$, $j=1,2,\dots,n$.
\end{itemize}

\begin{theorem}\label{thm31}
Let $(H_1)$-$(H_{3})$ hold. Suppose that
\begin{itemize}
\item[$(H_{4})$] there exists a positive constant $r$ such that
\begin{eqnarray*}
&&\max\limits_{1\leq i\leq n}\bigg\{\frac{\rho_{i}}{c_{i}^{-}},\frac{c_{i}^{+}+c_{i}^{-}}{c_{i}^{-}}\rho_{i}\bigg\}
+\max\big\{K_{1},K_{2}\big\}\leq r,\\
&&0<\frac{\rho_{i}}{r}<\frac{c_{i}^{-}}{c_{i}^{+}+c_{i}^{-}}<c_{i}^{-}
,\,\,\,i=1,2,\ldots,n,
\end{eqnarray*}
where
\begin{eqnarray*}
\rho_{i}&=&\big(c_{i}^{+}\delta_{i}^{+}+\sum_{j=1}^{n}a_{ij}^{+}L_{j}^{f}
+\sum_{j=1}^{n}b_{ij}^{+}L_{j}^{g}+\sum_{j=1}^{n}d_{ij}^{+}\sigma_{ij}^{+}L_{j}^{h}\\
&&+\sum_{j=1}^{n}\sum_{l=1}^{n}T_{ijl}^{+}L_{j}^{k}L_{l}^{k}\big)r,\,\,\,\,i=1,2,\ldots,n,\\
K_{1}&=&\max\limits_{1\leq i \leq
n}\bigg\{\frac{I_{i}^{+}}{c_{i}^{-}}\bigg\},\,\,\,\,K_{2}=\max\limits_{1\leq i\leq n}\bigg\{\frac{c_{i}^{+}+c_{i}^{-}}{c_{i}^{-}}I_{i}^{+}\bigg\}.
\end{eqnarray*}
\end{itemize}
Then system \eqref{a1} has at least one pseudo almost periodic solution in the region
$\mathbb{E}=\{\varphi\in\mathbb{B}: \|\varphi\|_{\mathbb{B}}\leq r\}$.
\end{theorem}
\begin{proof}
Rewrite \eqref{a1} in the form
\begin{eqnarray*}
x_{i}^{\nabla}(t)&=&-c_{i}(t)x_{i}(t)+c_{i}(t)\int_{t-\delta_{i}(t)}^{t}x_{i}^{\nabla}
(s) \nabla s+\sum_{j=1}^{n}a_{ij}(t)f_{j}(x_{j}(t))\\
&&+\sum_{j=1}^{n}b_{ij}(t)g_{j}(x_{j}(t-\tau_{ij}(t)))
+\sum_{j=1}^{n}d_{ij}(t)\int_{t-\sigma_{ij}(t)}^{t}h_{j}(x_{j}^{\nabla}(s))\nabla s\\
&&+\sum_{j=1}^{n}\sum_{l=1}^{n}T_{ijl}(t)k_{j}(x_{j}(t-\xi_{ijl}(t)))k_{l}(x_{l}(t-\zeta_{ijl}(t)))+I_{i}(t),\,\,\,t\in \mathbb{T},\,\,\, i=1,2,\ldots,n.
\end{eqnarray*}

For any $\varphi\in \mathbb{B}$, we consider the following system
\begin{eqnarray}\label{b1}
  x_{i}^{\nabla}(t)=-c_{i}(t)x_{i}(t)+F_{i}(t,\varphi)+I_{i}(t),\,\,\, t\in \mathbb{T},\quad i=1,2,\ldots,n,
\end{eqnarray}
where
\begin{eqnarray*}
F_{i}(t,\varphi)&=&c_{i}(t)\int_{t-\delta_{i}(t)}^{t}\varphi_{i}^{\nabla}
(s) \nabla s+\sum_{j=1}^{n}a_{ij}(t)f_{j}(\varphi_{j}(t))\\
&&+\sum_{j=1}^{n}b_{ij}(t)g_{j}(\varphi_{j}(t-\tau_{ij}(t)))
+\sum_{j=1}^{n}d_{ij}(t)\int_{t-\sigma_{ij}(t)}^{t}h_{j}(\varphi_{j}^{\nabla}(s))\nabla s\\
&&+\sum_{j=1}^{n}\sum_{l=1}^{n}T_{ijl}(t)k_{j}(\varphi_{j}(t-\xi_{ijl}(t)))k_{l}(\varphi_{l}(t-\zeta_{ijl}(t))).
\end{eqnarray*}
Since $\min\limits_{1\leq i\leq n}\big\{\inf\limits_{t\in\mathbb{T}}c_{i}(t)\big\}>0$, it follows from Lemma \ref{lem28} that the linear system
\begin{eqnarray}\label{b2}
  x_{i}^{\nabla}(t)=-c_{i}(t)x_{i}(t),\,\, i=1,2,\ldots,n
\end{eqnarray}
admits an exponential dichotomy on $\mathbb{T}$. Thus, by Lemma \ref{lem27}, we know that system \eqref{b1} has exactly one pseudo almost periodic solution which can be expressed as follows:
\[
x_\varphi=(x_{\varphi_1},x_{\varphi_2},\ldots,x_{\varphi_n})^{T},
\]
where
\begin{eqnarray*}
  x_{\varphi_{i}}(t)=\int_{-\infty}^{t}\hat{e}_{-c_{i}}(t,\rho(s))
  \big(F_{i}(s,\varphi)+I_{i}(s)\big)\nabla s,\,\, i=1,2,\ldots,n.
\end{eqnarray*}
 Define the following operator
\begin{eqnarray*}
\Phi:\mathbb{E} &\rightarrow &\mathbb{E}\\
(\varphi_{1},\varphi_{2},\ldots,\varphi_{n})^{T}&\rightarrow & (x_{\varphi_1},x_{\varphi_2},\ldots,x_{\varphi_n})^{T}.
\end{eqnarray*}
We will show that $\Phi$ is a contraction.

First, we show that for any $\varphi\in \mathbb{E}$, we have $\Phi\varphi\in \mathbb{E}$. Note that
\begin{eqnarray*}
|F_{i}(s,\varphi)|&=&\bigg|c_{i}(s)\int_{s-\delta_{i}(s)}^{s}\varphi_{i}^{\nabla}
(u) \nabla u+\sum_{j=1}^{n}a_{ij}(s)f_{j}(\varphi_{j}(s))\\
&&+\sum_{j=1}^{n}b_{ij}(s)g_{j}(\varphi_{j}(s-\tau_{ij}(s)))
+\sum_{j=1}^{n}d_{ij}(s)\int_{s-\sigma_{ij}(s)}^{s}h_{j}(\varphi_{j}^{\nabla}(u))\nabla u\\
&&+\sum_{j=1}^{n}\sum_{l=1}^{n}T_{ijl}(s)k_{j}(\varphi_{j}(s-\xi_{ijl}(s)))k_{l}(\varphi_{l}(s-\zeta_{ijl}(s)))\bigg|\\
  &\leq&c_{i}^{+}\bigg|\int_{s-\delta_{i}(s)}^{s}\varphi_{i}^{\nabla}(u) \nabla u\bigg|+\sum_{j=1}^{n}a_{ij}^{+}\big|f_{j}(\varphi_{j}(s))\big|\\
&&+\sum_{j=1}^{n}b_{ij}^{+}\big|g_{j}(\varphi_{j}(s-\tau_{ij}(s)))\big|
+\sum_{j=1}^{n}d_{ij}^{+}\bigg|\int_{s-\sigma_{ij}(s)}^{s}h_{j}(\varphi_{j}^{\nabla}(u))\nabla u\bigg|\\
&&+\sum_{j=1}^{n}\sum_{l=1}^{n}T_{ijl}^{+}\big|k_{j}(\varphi_{j}(s-\xi_{ijl}(s)))
k_{l}(\varphi_{l}(s-\zeta_{ijl}(s)))\big|\\
&\leq&c_{i}^{+}\delta_{i}^{+}|\varphi_{i}^{\nabla}(s)|+\sum_{j=1}^{n}a_{ij}^{+}L_{j}^{f}\big|\varphi_{j}(s)\big|
+\sum_{j=1}^{n}b_{ij}^{+}L_{j}^{g}\big|\varphi_{j}(s-\tau_{ij}(s))\big|\\
&&+\sum_{j=1}^{n}d_{ij}^{+}\sigma_{ij}^{+}L_{j}^{h}\big|\varphi_{j}^{\nabla}(s)\big|
+\sum_{j=1}^{n}\sum_{l=1}^{n}T_{ijl}^{+}L_{j}^{k}L_{l}^{k}\big|\varphi_{j}(s-\xi_{ijl}(s))\big|\big|\varphi_{l}(s-\zeta_{ijl}(s))\big|\\
&\leq&\bigg(c_{i}^{+}\delta_{i}^{+}+\sum_{j=1}^{n}a_{ij}^{+}L_{j}^{f}
+\sum_{j=1}^{n}b_{ij}^{+}L_{j}^{g}+\sum_{j=1}^{n}d_{ij}^{+}\sigma_{ij}^{+}L_{j}^{h}
+\sum_{j=1}^{n}\sum_{l=1}^{n}T_{ijl}^{+}L_{j}^{k}L_{l}^{k}\bigg)\|\varphi\|_{\mathbb{B}}\\
  &\leq&\bigg(c_{i}^{+}\delta_{i}^{+}+\sum_{j=1}^{n}a_{ij}^{+}L_{j}^{f}
+\sum_{j=1}^{n}b_{ij}^{+}L_{j}^{g}+\sum_{j=1}^{n}d_{ij}^{+}\sigma_{ij}^{+}L_{j}^{h}
+\sum_{j=1}^{n}\sum_{l=1}^{n}T_{ijl}^{+}L_{j}^{k}L_{l}^{k}\bigg)r\\
&=&\rho_{i},\,\,i=1,2,\ldots,n.
\end{eqnarray*}
Therefore, we can get
\begin{eqnarray*}
\sup\limits_{t\in\mathbb{T}}|x_{\varphi_{i}}(t)|&=&\sup\limits_{t\in\mathbb{T}}\bigg|\int_{-\infty}^{t}\hat{e}_{-c_{i}}(t,\rho(s))
  \big(F_{i}(s,\varphi)+I_{i}(s)\big)\nabla s\bigg|\\
  &\leq&\sup\limits_{t\in\mathbb{T}}\int_{-\infty}^{t}\hat{e}_{-c_{i}^{-}}(t,\rho(s))
  \big|F_{i}(s,\varphi)\big|\nabla s+\frac{I_{i}^{+}}{c_{i}^{-}}\\
  &\leq&\frac{\rho_{i}}{c_{i}^{-}}
  +\frac{I_{i}^{+}}{c_{i}^{-}}\leq r,\,\,\,\,i=1,2,\ldots,n.
\end{eqnarray*}
On the other hand, for $i=1,2,\ldots,n$, we have
\begin{eqnarray*}
  \sup\limits_{t\in\mathbb{T}}|x_{\varphi_{i}}^{\nabla}(t)|&=&\sup\limits_{t\in\mathbb{T}}\bigg|
  F_{i}(t,\varphi)+I_{i}(t)-c_{i}(t)\int_{-\infty}^{t}\hat{e}_{-c_{i}}(t,\rho(s))
  \big(F_{i}(s,\varphi)+I_{i}(s)\big)\nabla s \bigg|\\
 &\leq&\bigg(c_{i}^{+}\delta_{i}^{+}+\sum_{j=1}^{n}a_{ij}^{+}L_{j}^{f}
+\sum_{j=1}^{n}b_{ij}^{+}L_{j}^{g}+\sum_{j=1}^{n}d_{ij}^{+}\sigma_{ij}^{+}L_{j}^{h}
+\sum_{j=1}^{n}\sum_{l=1}^{n}T_{ijl}^{+}L_{j}^{k}L_{l}^{k}\bigg)r+I_{i}^{+}\\
 &&+c_{i}^{+}\int_{-\infty}^{t}\hat{e}_{-c_{i}^{-}}(t,\rho(s))
  \bigg(c_{i}^{+}\delta_{i}^{+}+\sum_{j=1}^{n}a_{ij}^{+}L_{j}^{f}
+\sum_{j=1}^{n}b_{ij}^{+}L_{j}^{g}+\sum_{j=1}^{n}d_{ij}^{+}\sigma_{ij}^{+}L_{j}^{h}\\
&&+\sum_{j=1}^{n}\sum_{l=1}^{n}T_{ijl}^{+}L_{j}^{k}L_{l}^{k}\bigg)r\nabla s+\frac{I_{i}^{+}}{c_{i}^{-}}\\
  &\leq&\frac{c_{i}^{+}+c_{i}^{-}}{c_{i}^{-}}\rho_{i}
  +\frac{c_{i}^{+}+c_{i}^{-}}{c_{i}^{-}}I_{i}^{+}
  \leq r.
\end{eqnarray*}
In view of $(H_{4})$, we have
\begin{eqnarray*}
\|\Phi(\varphi)\|_{\mathbb{B}}=\max\limits_{1\leq i\leq n}\bigg\{\sup\limits_{t\in\mathbb{T}}|x_{\varphi_{i}}(t)|, \sup\limits_{t\in\mathbb{T}}|x_{\varphi_{i}}^{\nabla}(t)|\bigg\}\leq r,
\end{eqnarray*}
which implies that $\Phi\varphi \in \mathbb{E}$. Hence, the mapping $\Phi$ is
a self-mapping from $\mathbb{E}$ to $\mathbb{E}$. Next, we shall prove that $\Phi$ is a contraction mapping.
For any $\varphi,\psi\in \mathbb{E}$, we denote
\begin{eqnarray*}
H_{i}(s,\varphi,\psi)
&=&c_{i}(s)\int_{s-\delta_{i}(s)}^{s}[\varphi_{i}^{\nabla}(u)-\psi_{i}^{\nabla}(u) ] \nabla u
+\sum_{j=1}^{n}a_{ij}(s)\big[f_{j}(\varphi_{j}(s))-f_{j}(\psi_{j}(s))\big]\\
&&+\sum_{j=1}^{n}b_{ij}(s)\big[g_{j}(\varphi_{j}(s-\tau_{ij}(s)))-g_{j}(\psi_{j}(s-\tau_{ij}(s)))\big]\\
&&+\sum_{j=1}^{n}d_{ij}(s)\int_{s-\sigma_{ij}(s)}^{s}\big[h_{j}(\varphi_{j}^{\nabla}(u))-h_{j}(\psi_{j}^{\nabla}(u))\big]\nabla u\\
&&+\sum_{j=1}^{n}\sum_{l=1}^{n}T_{ijl}(s)\big[k_{j}(\varphi_{j}(s-\xi_{ijl}(s)))k_{l}(\varphi_{l}(s-\zeta_{ijl}(s)))\\
&&-k_{j}(\psi_{j}(s-\xi_{ijl}(s)))k_{l}(\psi_{l}(s-\zeta_{ijl}(s)))\big], i=1,2,\ldots,n.
\end{eqnarray*}
Then, for $i=1,2,\ldots,n$, we have
\begin{eqnarray*}
\sup\limits_{t\in\mathbb{T}}\big|x_{\varphi_{i}}(t)-x_{\psi_{i}}(t)\big|&=&
\sup\limits_{t\in\mathbb{T}}\bigg|\int_{-\infty}^{t}\hat{e}_{-c_{i}}(t,\rho(s))
  H_{i}(s,\varphi,\psi)\nabla s\bigg|\\
  &\leq&\sup\limits_{t\in\mathbb{T}}\int_{-\infty}^{t}\hat{e}_{-c_{i}^{-}}(t,\rho(s))
  \bigg(c_{i}^{+}\delta_{i}^{+}+\sum_{j=1}^{n}a_{ij}^{+}L_{j}^{f}+\sum_{j=1}^{n}b_{ij}^{+}L_{j}^{g}\\
&&+\sum_{j=1}^{n}d_{ij}^{+}\sigma_{ij}^{+}L_{j}^{h}
+\sum_{j=1}^{n}\sum_{l=1}^{n}T_{ijl}^{+}L_{j}^{k}L_{l}^{k}\bigg)\nabla s\|\varphi-\psi\|_{\mathbb{B}}\\
  &\leq&\frac{1}{c_{i}^{-}}\big(c_{i}^{+}\delta_{i}^{+}+\sum_{j=1}^{n}a_{ij}^{+}L_{j}^{f}
+\sum_{j=1}^{n}b_{ij}^{+}L_{j}^{g}+\sum_{j=1}^{n}d_{ij}^{+}\sigma_{ij}^{+}L_{j}^{h}\\
&&+\sum_{j=1}^{n}\sum_{l=1}^{n}T_{ijl}^{+}L_{j}^{k}L_{l}^{k}\big)\|\varphi-\psi\|_{\mathbb{B}},\\
\sup\limits_{t\in\mathbb{T}}\big|(x_{\varphi_{i}}(t)-x_{\psi_{i}}(t))^{\nabla}\big|
&=&\sup\limits_{t\in\mathbb{T}}\bigg|\bigg(\int_{-\infty}^{t}\hat{e}_{-c_{i}}(t,\rho(s))
  H_{i}(s,\varphi,\psi)\nabla s\bigg)^{\nabla}\bigg|\\
&=&\sup\limits_{t\in\mathbb{T}}\bigg|H_{i}(t,\varphi,\psi)-c_{i}(t)\int_{-\infty}^{t}\hat{e}_{-c_{i}}(t,\rho(s))
  H_{i}(s,\varphi,\psi)\nabla s\bigg|\\
&\leq&|H_{i}(t,\varphi,\psi)|+c_{i}^{+}\sup\limits_{t\in\mathbb{T}}\bigg|\int_{-\infty}^{t}\hat{e}_{-c_{i}^{-}}(t,\rho(s))
  H_{i}(s,\varphi,\psi)\nabla s\bigg|\\
  &\leq&\bigg(c_{i}^{+}\delta_{i}^{+}+\sum_{j=1}^{n}a_{ij}^{+}L_{j}^{f}
+\sum_{j=1}^{n}b_{ij}^{+}L_{j}^{g}+\sum_{j=1}^{n}d_{ij}^{+}\sigma_{ij}^{+}L_{j}^{h}\\
&&+\sum_{j=1}^{n}\sum_{l=1}^{n}T_{ijl}^{+}L_{j}^{k}L_{l}^{k}\bigg)\|\varphi-\psi\|_{\mathbb{B}}\\
 &&+\frac{c_{i}^{+}}{c_{i}^{-}}\bigg(c_{i}^{+}\delta_{i}^{+}+\sum_{j=1}^{n}a_{ij}^{+}L_{j}^{f}
+\sum_{j=1}^{n}b_{ij}^{+}L_{j}^{g}+\sum_{j=1}^{n}d_{ij}^{+}\sigma_{ij}^{+}L_{j}^{h}\\
&&+\sum_{j=1}^{n}\sum_{l=1}^{n}T_{ijl}^{+}L_{j}^{k}L_{l}^{k}\bigg)\|\varphi-\psi\|_{\mathbb{B}}\\
 &=&\frac{c_{i}^{+}+c_{i}^{-}}{c_{i}^{-}}\bigg(c_{i}^{+}\delta_{i}^{+}+\sum_{j=1}^{n}a_{ij}^{+}L_{j}^{f}
+\sum_{j=1}^{n}b_{ij}^{+}L_{j}^{g}+\sum_{j=1}^{n}d_{ij}^{+}\sigma_{ij}^{+}L_{j}^{h}\\
&&+\sum_{j=1}^{n}\sum_{l=1}^{n}T_{ijl}^{+}L_{j}^{k}L_{l}^{k}\bigg)\|\varphi-\psi\|_{\mathbb{B}}.
\end{eqnarray*}
By $(H_{4})$, we have
\begin{eqnarray*}
\|\Phi(\varphi)-\Phi(\psi)\|_{\mathbb{B}}<\|\varphi-\psi\|_{\mathbb{B}}.
\end{eqnarray*}
Hence, we obtain that $\Phi$ is a contraction mapping. By the fixed point theorem of Banach space \cite{41}, it follows that $\Phi$
has a fixed point in $\mathbb{E}$; that is, system \eqref{a1} has a unique pseudo almost periodic solution. This completes the proof of Theorem \ref{thm31}.
\end{proof}

\section{Global exponential stability of pseudo almost periodic solution}

\setcounter{equation}{0}

\indent

In this section, we will study the exponential stability of pseudo almost periodic solutions of \eqref{a1}.

\begin{definition}
The pseudo almost periodic solution $x^{\ast}(t)=(x_{1}^{\ast}(t),x_{2}^{\ast}(t),\ldots,x_{n}^{\ast}(t))^{T}$ of
system \eqref{a1} with initial value $\varphi^{\ast}(t)=(\varphi_{1}^{\ast}(t),\varphi_{2}^{\ast}(t),\ldots,\varphi_{n}^{\ast}(t))^{T}$ is said to be globally exponentially stable if there exist a positive constant $\lambda$ with $\ominus_{\nu}\lambda\in \mathcal{R}^+$ and $M>1$ such that every solution $x(t)=(x_{1}(t),x_{2}(t),\ldots,x_{n}(t))^{T}$ of system \eqref{a1} with initial value
$\varphi(t)=(\varphi_{1}(t),\varphi_{2}(t),\ldots,\varphi_{n}(t))^{T}$
satisfies
\[\|x(t)-x^{\ast}(t)\|\leq Me_{\ominus_{\nu}\lambda}(t,t_0)\|\psi\|,\,\,\,\,\forall t\in(0,+\infty)_{\mathbb{T}},\]
where $\|\psi\|=\sup\limits_{t\in[-\theta,0]_{\mathbb{T}}}\max\limits_{1\leq i\leq n}|\varphi_{i}(t)-\varphi_{i}^{\ast}(t)|$, $t_{0}=\max\{[-\theta,0]_{\mathbb{T}}\}$.
\end{definition}

\begin{theorem}\label{thm41}
Assume that  $(H_1)$-$(H_4)$ hold, then system \eqref{a1} has a unique
almost periodic solution that is globally exponentially stable.
\end{theorem}
\begin{proof}
From Theorem \ref{thm31},  we see that system \eqref{a1} has at least one pseudo almost periodic solution $x^{\ast}(t)=(x_{1}^{\ast}(t),x_{2}^{\ast}(t),\ldots,x_{n}^{\ast}(t))^{T}$ with
initial value $\varphi^{\ast}(t)=(\varphi_{1}^{\ast}(t),\varphi_{2}^{\ast}(t),\ldots,\varphi_{n}^{\ast}(t))^{T}$. Suppose that $x(t)=(x_{1}(t),x_{2}(t),\ldots,x_{n}(t))^{T}$ is an arbitrary solution of \eqref{a1} with initial value
$\varphi(t)=(\varphi_{1}(t),\varphi_{2}(t),\ldots,\varphi_{n}(t))^{T}$.
Then it follows from system \eqref{a1} that
\begin{eqnarray}\label{c1}
z_{i}^{\nabla}(t)&=&-c_{i}(t)z_{i}(t-\delta_{i}(t))+\sum_{j=1}^{n}a_{ij}(t)\big[f_{j}(x_{j}(t))-f_{j}(x^{\ast}_{j}(t))\big]
+\sum_{j=1}^{n}b_{ij}(t)\big[g_{j}(x_{j}(t-\tau_{ij}(t)))\nonumber\\
&&-g_{j}(x^{\ast}_{j}(t-\tau_{ij}(t)))\big]
+\sum_{j=1}^{n}d_{ij}(t)\int_{t-\sigma_{ij}(t)}^{t}\big[h_{j}(x_{j}^{\nabla}(s))-h_{j}((x^{\ast}_{j})^{\nabla}(s))\big]\nabla s\nonumber\\
&&+\sum_{j=1}^{n}\sum_{l=1}^{n}T_{ijl}(t)\big[k_{j}(x_{j}(t-\xi_{ijl}(t)))k_{l}(x_{l}(t-\zeta_{ijl}(t)))\nonumber\\
&&-k_{j}(x^{\ast}_{j}(t-\xi_{ijl}(t)))k_{l}(x^{\ast}_{l}(t-\zeta_{ijl}(t)))\big],
\end{eqnarray}
where $u_{i}(t)=x_{i}(t)-x^{\ast}_{i}(t)$ and $i=1,2,\ldots,n$.

The initial condition of \eqref{c1} is
\begin{eqnarray*}
\psi_{i}(s)=\varphi_{i}(s)-\varphi_{i}^{\ast}(s),\,\,\,\,\, \psi^{\nabla}_{j}(s)=\varphi^{\Delta}_{i}(s)-(\varphi_{i}^{\ast})^{\nabla}(s),
\end{eqnarray*}
where $s\in[-\theta,0]_{\mathbb{T}}$, $i=1,2,\dots,n$.

Rewrite \eqref{c1} in the form
\begin{eqnarray}\label{c2}
z_{i}^{\nabla}(t)+c_{i}(t)z_{i}(t)&=&c_{i}(t)\int_{t-\delta_{i}(t)}^{t}z_{i}^{\nabla}(s)
\nabla s+\sum_{j=1}^{n}a_{ij}(t)\big[f_{j}(x_{j}(t))-f_{j}(x^{\ast}_{j}(t))\big]\nonumber\\
&&+\sum_{j=1}^{n}b_{ij}(t)\big[g_{j}(x_{j}(t-\tau_{ij}(t)))
-g_{j}(x^{\ast}_{j}(t-\tau_{ij}(t)))\big]\nonumber\\
&&+\sum_{j=1}^{n}d_{ij}(t)\int_{t-\sigma_{ij}(t)}^{t}\big[h_{j}(x_{j}^{\nabla}(s))-h_{j}((x^{\ast}_{j})^{\nabla}(s))\big]\nabla s\nonumber\\
&&+\sum_{j=1}^{n}\sum_{l=1}^{n}T_{ijl}(t)\big[k_{j}(x_{j}(t-\xi_{ijl}(t)))k_{l}(x_{l}(t-\zeta_{ijl}(t)))\nonumber\\
&&-k_{j}(x^{\ast}_{j}(t-\xi_{ijl}(t)))k_{l}(x^{\ast}_{l}(t-\zeta_{ijl}(t)))\big],\,\,i=1,2,\dots,n.
\end{eqnarray}
Multiplying the both sides of \eqref{c2} by $\hat{e}_{-c_{i}}(t,\rho(s))$ and
integrating over $[t_{0},t]_{\mathbb{T}}$, where $t_{0}\in[-\theta,0]_{\mathbb{T}}$, by Lemma 2.4, we get
\begin{eqnarray}\label{c5}
 z_{i}(t)&=& z_{i}(t_{0})\hat{e}_{-c_{i}}(t,t_{0})
 +\int_{t_{0}}^{t}\hat{e}_{-c_{i}}(t,\rho(s))\bigg(c_{i}(s)\int_{s-\delta_{i}(s)}^{s}z_{i}^{\nabla}(u)
\nabla u\nonumber\\
&&+\sum_{j=1}^{n}a_{ij}(s)\big[f_{j}(x_{j}(s))-f_{j}(x^{\ast}_{j}(s))\big]
+\sum_{j=1}^{n}b_{ij}(s)\big[g_{j}(x_{j}(s-\tau_{ij}(s)))\nonumber\\
&&-g_{j}(x^{\ast}_{j}(s-\tau_{ij}(s)))\big]
+\sum_{j=1}^{n}d_{ij}(s)\int_{s-\sigma_{ij}(s)}^{s}\big[h_{j}(x_{j}^{\nabla}(u))-h_{j}((x^{\ast}_{j})^{\nabla}(u))\big]\nabla u\nonumber\\
&&+\sum_{j=1}^{n}\sum_{l=1}^{n}T_{ijl}(s)\big[k_{j}(x_{j}(s-\xi_{ijl}(s)))k_{l}(x_{l}(s-\zeta_{ijl}(s)))\nonumber\\
&&-k_{j}(x^{\ast}_{j}(s-\xi_{ijl}(s)))k_{l}(x^{\ast}_{l}(s-\zeta_{ijl}(s)))\big]\bigg)\nabla s,\,\,\,i=1,2,\ldots,n.
\end{eqnarray}

Let $R_{i}$ be defined as follows:
\begin{eqnarray*}
R_{i}(\beta)&=&c_{i}^{-}-\beta-\big(c_{i}^{+}\exp(\beta\sup\limits_{s\in
\mathbb{T}}\nu(s))+c_{i}^{-}-\beta\big)
\bigg(c_{i}^{+}\delta_{i}^{+}\exp(\beta\delta_{i}^{+})\\
&&+\sum_{j=1}^{n}a_{ij}^{+}L_{j}^{f}
+\sum_{j=1}^{n}b_{ij}^{+}L_{j}^{g}\exp(\beta\tau_{ij}^{+})
+\sum_{j=1}^{n}d_{ij}^{+}L_{j}^{h}\sigma_{ij}^{+}\exp(\beta\sigma_{ij}^{+})\\
&&+\sum_{j=1}^{n}\sum_{l=1}^{n}T_{ijl}^{+}L_{j}^{k}L_{l}^{k}\big(\exp(\beta\xi_{ijl}^{+})
+\exp(\beta\zeta_{ijl}^{+})\big)\bigg),\,\,i=1,2,\ldots,n.
\end{eqnarray*}
By $(H_{4})$, we get
\begin{eqnarray*}
R_{i}(0)&=&c_{i}^{-}-\bigg(c_{i}^{+}\delta_{i}^{+}+\sum_{j=1}^{n}a_{ij}^{+}L_{j}^{f}
+\sum_{j=1}^{n}b_{ij}^{+}L_{j}^{g}
+\sum_{j=1}^{n}d_{ij}^{+}L_{j}^{h}\\
&&+\sum_{j=1}^{n}\sum_{l=1}^{n}T_{ijl}^{+}L_{j}^{k}L_{l}^{k}\bigg)> 0,
\,\,i=1,2,\ldots,n.
\end{eqnarray*}
Since $R_{i}$ is continuous on $[0,+\infty)$ and
$R_{i}(\beta) \rightarrow -\infty$, as
$\beta\rightarrow +\infty$,  there exists $\gamma_{i} > 0$  such that
$R_{i}(\gamma_{i})=0$ and $R_{i}(\beta)> 0$
for $\beta\in(0,\gamma_{i})$. Take
$a=\min\limits_{1\leq i\leq n}\big\{\gamma_{i}\big\}$,
we have $R_{i}(a)\geq 0$. So, we can
choose a positive constant $0< \lambda <
\min\big\{a,\min\limits_{1\leq i \leq n}\{c_{i}^{-}\}\big\}$ such
that
\[
R_{i}(\lambda)>0, \quad i=1,2,\ldots,n,
\]
which implies that
\begin{eqnarray}\label{e42}
&&\bigg(1+\frac{c_{i}^{+}\exp(\lambda\sup\limits_{s\in
\mathbb{T}}\nu(s))}{c_{i}^{-}-\lambda}\bigg)
\bigg(c_{i}^{-}-\bigg(c_{i}^{+}\delta_{i}^{+}+\sum_{j=1}^{n}a_{ij}^{+}L_{j}^{f}
+\sum_{j=1}^{n}b_{ij}^{+}L_{j}^{g}
+\sum_{j=1}^{n}d_{ij}^{+}L_{j}^{h}\nonumber\\
&&+\sum_{j=1}^{n}\sum_{l=1}^{n}T_{ijl}^{+}L_{j}^{k}L_{l}^{k}\bigg)\bigg)< 1,\,\,i=1,2,\ldots,n.
\end{eqnarray}

Let
\begin{eqnarray*}
M=\max\limits_{1\leq i \leq n}\bigg\{\frac{c_{i}^{-}}{c_{i}^{-}-\bigg(c_{i}^{+}\delta_{i}^{+}+\sum_{j=1}^{n}a_{ij}^{+}L_{j}^{f}
+\sum_{j=1}^{n}b_{ij}^{+}L_{j}^{g}
+\sum_{j=1}^{n}d_{ij}^{+}L_{j}^{h}
+\sum_{j=1}^{n}\sum_{l=1}^{n}T_{ijl}^{+}L_{j}^{k}L_{l}^{k}\bigg)}\bigg\},
\end{eqnarray*}
then by $(H_{4})$ we have $M>1$.

Hence, it is obvious that
\begin{eqnarray}
\|z(t)\|_{\mathbb{B}}\leq M\hat{e}_{\ominus_{\nu}\lambda}(t,t_{0})\|\psi\|_{\mathbb{B}}, \quad \forall
t\in[-\theta,t_{0}]_{\mathbb{T}},
\end{eqnarray}
where $\ominus_{\nu}\lambda\in\mathcal{R}^{+}$. We claim that
\begin{eqnarray}\label{e45}
\|z(t)\|_{\mathbb{B}}\leq M\hat{e}_{\ominus_{\nu}\lambda}(t,t_{0})\|\psi\|_{\mathbb{B}}, \quad \forall
t\in(t_{0},+\infty)_{\mathbb{T}}.
\end{eqnarray}
To prove \eqref{e45}, we show that for any $P>1$, the following inequality holds:
\begin{eqnarray}\label{e46}
\|z(t)\|_{\mathbb{B}}< PM\hat{e}_{\ominus_{\nu}\lambda}(t,t_{0})\|\psi\|_{\mathbb{B}}, \quad \forall
t\in(t_{0},+\infty)_{\mathbb{T}},
\end{eqnarray}
which implies that, for $i=1,2,\ldots,n$, we have
\begin{eqnarray}\label{e47}
|z_{i}(t)|< PM\hat{e}_{\ominus_{\nu}\lambda}(t,t_{0})\|\psi\|_{\mathbb{B}}, \quad \forall
t\in(t_{0},+\infty)_{\mathbb{T}}
\end{eqnarray}
and
\begin{eqnarray}\label{e48}
|z^{\nabla}_{i}(t)|< PM\hat{e}_{\ominus_{\nu}\lambda}(t,t_{0})\|\psi\|_{\mathbb{B}}, \quad \forall
t\in(t_{0},+\infty)_{\mathbb{T}}.
\end{eqnarray}
If \eqref{e46} is not true, then there must be some $t_{1}\in
(t_{0},+\infty)_{\mathbb{T}}$ and some $i_1, i_2\in \{1,2,\ldots,n\}$
such that
\begin{eqnarray*}
\|z(t_1)\|_{\mathbb{B}}=\max\{|z_{i_{1}}(t_{1})|,|z_{i_{2}}^{\nabla}(t_1)|\}
\geq P M\hat{e}_{\ominus_{\nu}\lambda}(t_1,t_0)\|\psi\|_{\mathbb{B}}
\end{eqnarray*}
and
\begin{eqnarray*}
\|z(t)\|_{\mathbb{B}} \leq PM\hat{e}_{\ominus_{\nu}\lambda}(t,t_{0})\|\psi\|_{\mathbb{B}},\quad
t\in(t_{0},t_{1}]_{\mathbb{T}},\,\,t_{0}\in[-\theta,0]_{\mathbb{T}}.
\end{eqnarray*}
Therefore, there must exist a constant $c\geq 1$ such that
\begin{eqnarray}\label{e48}
\|z(t_{1})\|_{\mathbb{B}}&=&\max\{|z_{i_1}(t_1)|,|z_{i_2}^{\nabla}(t_1)|\}=cP M \hat{e}_{\ominus_{\nu}\lambda}(t_1,t_0)\|\psi\|_{\mathbb{B}}
\end{eqnarray}
and
\begin{eqnarray}\label{e49}
\|z(t)\|_{\mathbb{B}} \leq cP M \hat{e}_{\ominus_{\nu}\lambda}(t,t_{0})\|\psi\|_{\mathbb{B}},\quad
t\in(t_{0},t_{1}]_{\mathbb{T}},\,\,t_{0}\in[-\theta,0]_{\mathbb{T}}.
\end{eqnarray}
In view of \eqref{c5}, we have
\begin{eqnarray*}
&&|z_{i_{1}}(t_{1})| \\
&=&\bigg|z_{i_{1}}(t_{0})e_{-c_{i_{1}}}(t_{1},t_{0})+\int_{t_{0}}^{t_{1}}e_{-c_{i_{1}}}(t_{1},\rho(s))\bigg(
c_{i_{1}}(s)\int_{s-\delta_{i_{1}}(s)}^{s}z_{i_{1}}^{\nabla}(u)
\nabla u\\
&&+\sum_{j=1}^{n}a_{i_{1}j}(s)\big[f_{j}(x_{j}(s))-f_{j}(x^{\ast}_{j}(s))\big]
+\sum_{j=1}^{n}b_{i_{1}j}(s)\big[g_{j}(x_{j}(s-\tau_{i_{1}j}(s)))\\
&&-g_{j}(x^{\ast}_{j}(s-\tau_{i_{1}j}(s)))\big]
+\sum_{j=1}^{n}d_{i_{1}j}(s)\int_{s-\sigma_{i_{1}j}(s)}^{s}\big[h_{j}(x_{j}^{\nabla}(u))-h_{j}((x^{\ast}_{j})^{\nabla}(u))\big]\nabla u\nonumber\\
&&+\sum_{j=1}^{n}\sum_{l=1}^{n}T_{i_{1}jl}(s)\big[k_{j}(x_{j}(s-\xi_{i_{1}jl}(s)))k_{l}(x_{l}(s-\zeta_{i_{1}jl}(s)))\\
&&-k_{j}(x^{\ast}_{j}(s-\xi_{i_{1}jl}(s)))k_{l}(x^{\ast}_{l}(s-\zeta_{i_{1}jl}(s)))\big]\bigg)\nabla s\bigg|\\
&\leq&\hat{e}_{-c_{i_{1}}}(t_{1},t_{0})\|\psi\|_{\mathbb{B}}
+cPM\hat{e}_{\ominus_{\nu}\lambda}(t_{1},t_{0})\|\psi\|_{\mathbb{B}}\int_{t_{0}}^{t_{1}}\hat{e}_{-c_{i_{1}}}(t_{1},\rho(s))
\hat{e}_{\lambda}(t_{1},\rho(s))\\
&&\times\bigg(c_{i_{1}}^{+}\int_{s-\delta_{i_{1}}(s)}^{s}\hat{e}_{\lambda}(\rho(u),u) \nabla u+\sum_{j=1}^{n}a_{i_{1}j}^{+}L_{j}^{f}\hat{e}_{\lambda}(\rho(s),s)
+\sum_{j=1}^{n}b_{i_{1}j}^{+}L_{j}^{g}\hat{e}_{\lambda}(\rho(s),s-\tau_{i_{1}j}(s))\\
&&+\sum_{j=1}^{n}d_{i_{1}j}^{+}L_{j}^{h}\int_{s-\sigma_{i_{1}j}(s)}^{s}\hat{e}_{\lambda}(\rho(u),u)\nabla u+\sum_{j=1}^{n}\sum_{l=1}^{n}T_{i_{1}jl}^{+}L_{j}^{k}L_{l}^{k}\big(\hat{e}_{\lambda}(\rho(s),s-\xi_{i_{1}jl}(s))\\
&&+\hat{e}_{\lambda}(\rho(s),s-\zeta_{i_{1}jl}(s))\big)\bigg)\nabla s\\
&\leq&\hat{e}_{-c_{i_{1}}}(t_{1},t_{0})\|\psi\|_{\mathbb{B}}
+cPM\hat{e}_{\ominus_{\nu}\lambda}(t_{1},t_{0})\|\psi\|_{\mathbb{B}}\int_{t_{0}}^{t_{1}}
\hat{e}_{-c_{i_{1}}\oplus_{\nu}\lambda}(t_{1},\rho(s))\\
&&\times\bigg(c_{i_{1}}^{+}\delta_{i_{1}}^{+}\hat{e}_{\lambda}(\rho(s),s-\delta_{i_{1}}(s))
+\sum_{j=1}^{n}a_{i_{1}j}^{+}L_{j}^{f}\hat{e}_{\lambda}(\rho(s),s)
+\sum_{j=1}^{n}b_{i_{1}j}^{+}L_{j}^{g}\hat{e}_{\lambda}(\rho(s),s-\tau_{i_{1}j}(s))\\
&&+\sum_{j=1}^{n}d_{i_{1}j}^{+}L_{j}^{h}\sigma_{i_{1}j}^{+}\hat{e}_{\lambda}(\rho(s),s-\sigma_{i_{1}j}(s))
+\sum_{j=1}^{n}\sum_{l=1}^{n}T_{i_{1}jl}^{+}L_{j}^{k}L_{l}^{k}\big(\hat{e}_{\lambda}(\rho(s),s-\xi_{i_{1}jl}(s))\\
&&+\hat{e}_{\lambda}(\rho(s),s-\zeta_{i_{1}jl}(s))\big)\bigg)\nabla s\\
  &\leq&\hat{e}_{-c_{i_{1}}}(t_{1},t_{0})\|\psi\|_{\mathbb{B}}
+cPM\hat{e}_{\ominus_{\nu}\lambda}(t_{1},t_{0})\|\psi\|_{\mathbb{B}}
\int_{t_{0}}^{t_{1}}\hat{e}_{-c_{i_{1}}\oplus_{\nu}\lambda}(t_{1},\rho(s))\\
&&\times\bigg(c_{i_{1}}^{+}\delta_{i_{1}}^{+}\exp\big[\lambda(\delta_{i_{1}}^{+}+\sup\limits_{s\in\mathbb{T}}\nu(s))\big]
+\sum_{j=1}^{n}a_{i_{1}j}^{+}L_{j}^{f}\exp\big(\lambda\sup\limits_{s\in\mathbb{T}}\nu(s)\big)\\
&&+\sum_{j=1}^{n}b_{i_{1}j}^{+}L_{j}^{g}\exp\big[\lambda(\tau_{i_{1}j}^{+}+\sup\limits_{s\in\mathbb{T}}\nu(s))\big]
+\sum_{j=1}^{n}d_{i_{1}j}^{+}L_{j}^{h}\sigma_{i_{1}j}^{+}\exp\big[\lambda(\sigma_{i_{1}j}^{+}+\sup\limits_{s\in\mathbb{T}}\nu(s))\big]\\
&&+\sum_{j=1}^{n}\sum_{l=1}^{n}T_{i_{1}jl}^{+}L_{j}^{k}L_{l}^{k}
\big(\exp\big[\lambda(\xi_{i_{1}j}^{+}+\sup\limits_{s\in\mathbb{T}}\nu(s))\big]
+\exp\big[\lambda(\zeta_{i_{1}j}^{+}+\sup\limits_{s\in\mathbb{T}}\nu(s))\big]\big)\bigg)\nabla s\\
&=&cPM \hat{e}_{\ominus_{\nu}\lambda}(t_{1},t_{0})\|\psi\|_{\mathbb{B}}\bigg\{\frac{1}{p
M}\hat{e}_{-c_{i_{1}}\oplus_{\nu}\lambda}(t_{1},t_{0})+\exp\big(\lambda\sup\limits_{s\in\mathbb{T}}\nu(s)\big)\\
&&\times\bigg(c_{i_{1}}^{+}\delta_{i_{1}}^{+}\exp(\lambda\delta_{i_{1}}^{+})
+\sum_{j=1}^{n}a_{i_{1}j}^{+}L_{j}^{f}+\sum_{j=1}^{n}b_{i_{1}j}^{+}L_{j}^{g}\exp(\lambda\tau_{i_{1}j}^{+})\\
&&+\sum_{j=1}^{n}d_{i_{1}j}^{+}L_{j}^{h}\sigma_{i_{1}j}^{+}\exp(\lambda\sigma_{i_{1}j}^{+})
+\sum_{j=1}^{n}\sum_{l=1}^{n}T_{i_{1}jl}^{+}L_{j}^{k}L_{l}^{k}
\big(\exp(\lambda\xi_{i_{1}j}^{+})
+\exp(\lambda\zeta_{i_{1}j}^{+})\big)\bigg)\\
&&\times\int_{t_{0}}^{t_{1}}\hat{e}_{-c_{i_{1}}\oplus_{\nu}\lambda}(t_{1},\rho(s))\bigg\}\nabla s\\
&\leq&cPM \hat{e}_{\ominus_{\nu}\lambda}(t_{1},t_{0})\|\psi\|_{\mathbb{B}}\bigg\{\frac{1}{p
M}\hat{e}_{-c_{i_{1}}\oplus_{\nu}\lambda}(t_{1},t_{0})+\exp\big(\lambda\sup\limits_{s\in\mathbb{T}}\nu(s)\big)\\
&&\times\bigg(c_{i_{1}}^{+}\delta_{i_{1}}^{+}\exp(\lambda\delta_{i_{1}}^{+})
+\sum_{j=1}^{n}a_{i_{1}j}^{+}L_{j}^{f}+\sum_{j=1}^{n}b_{i_{1}j}^{+}L_{j}^{g}\exp(\lambda\tau_{i_{1}j}^{+})\\
&&+\sum_{j=1}^{n}d_{i_{1}j}^{+}L_{j}^{h}\sigma_{i_{1}j}^{+}\exp(\lambda\sigma_{i_{1}j}^{+})
+\sum_{j=1}^{n}\sum_{l=1}^{n}T_{i_{1}jl}^{+}L_{j}^{k}L_{l}^{k}
\big(\exp(\lambda\xi_{i_{1}j}^{+})
+\exp(\lambda\zeta_{i_{1}j}^{+})\big)\bigg)\\
&&\times \frac{1-\hat{e}_{-c_{i_{1}}\oplus_{\nu}\lambda}(t_{1},t_{0})}{c_{i_{1}}^{-}-\lambda}\bigg\}\\
&\leq&cPM \hat{e}_{\ominus_{\nu}\lambda}(t_{1},t_{0})\|\psi\|_{\mathbb{B}}\bigg\{
\bigg[\frac{1}{M}-\frac{\exp\big(\lambda\sup\limits_{s\in\mathbb{T}}\nu(s)\big)}{c_{i_{1}}^{-}-\lambda}\bigg(
c_{i_{1}}^{+}\delta_{i_{1}}^{+}\exp(\lambda\delta_{i_{1}}^{+})\\
&&+\sum_{j=1}^{n}a_{i_{1}j}^{+}L_{j}^{f}+\sum_{j=1}^{n}b_{i_{1}j}^{+}L_{j}^{g}\exp(\lambda\tau_{i_{1}j}^{+})
+\sum_{j=1}^{n}d_{i_{1}j}^{+}L_{j}^{h}\sigma_{i_{1}j}^{+}\exp(\lambda\sigma_{i_{1}j}^{+})\\
&&+\sum_{j=1}^{n}\sum_{l=1}^{n}T_{i_{1}jl}^{+}L_{j}^{k}L_{l}^{k}
\big(\exp(\lambda\xi_{i_{1}j}^{+})
+\exp(\lambda\zeta_{i_{1}j}^{+})\big)\bigg)\bigg]\hat{e}_{-c_{i_{1}}\oplus_{\nu}\lambda}(t_{1},t_{0})\\
&&+\frac{\exp\big(\lambda\sup\limits_{s\in\mathbb{T}}\nu(s)\big)}{c_{i_{1}}^{-}-\lambda}
\bigg(c_{i_{1}}^{+}\delta_{i_{1}}^{+}\exp(\lambda\delta_{i_{1}}^{+})+
\sum_{j=1}^{n}a_{i_{1}j}^{+}L_{j}^{f}+\sum_{j=1}^{n}b_{i_{1}j}^{+}L_{j}^{g}\exp(\lambda\tau_{i_{1}j}^{+})\\
&&+\sum_{j=1}^{n}d_{i_{1}j}^{+}L_{j}^{h}\sigma_{i_{1}j}^{+}\exp(\lambda\sigma_{i_{1}j}^{+})
+\sum_{j=1}^{n}\sum_{l=1}^{n}T_{i_{1}jl}^{+}L_{j}^{k}L_{l}^{k}
\big(\exp(\lambda\xi_{i_{1}j}^{+})
+\exp(\lambda\zeta_{i_{1}j}^{+})\big)\bigg)
\bigg\}\\
&\leq&cPM \hat{e}_{\ominus_{\nu}\lambda}(t_{1},t_{0})\|\psi\|_{\mathbb{B}}
\end{eqnarray*}
and
\begin{eqnarray*}
|z_{i_{2}}^{\nabla}(t_{1})|&\leq& c_{i_2}^{+} \hat{e}_{-c_{i_2}}(t_{1},t_{0})\|\psi\|_{\mathbb{B}}+cP
M\hat{e}_{\ominus_{\nu}\lambda}(t_{1},t_{0})\|\psi\|_{\mathbb{B}}
\bigg(c_{i_2}^{+}\int_{t_{1}-\delta_{i_2}(t_{1})}^{t_{1}}e_{\lambda}(t_{1},u)
\nabla u\\
&&+\sum_{j=1}^{n}a_{i_2j}^{+}L_{j}^{f}\hat{e}_{\lambda}(t_1,t_1)
+\sum_{j=1}^{n}b_{i_2j}^{+}L_{j}^{g}\hat{e}_{\lambda}(t_1,t_1-\tau_{i_2j}(t_1))\\
&&+\sum_{j=1}^{n}d_{i_{2}j}^{+}L_{j}^{h}\int_{s-\sigma_{i_{2}j}(s)}^{s}\hat{e}_{\lambda}(\rho(u),u)\nabla u\\
&&+\sum_{j=1}^{n}\sum_{l=1}^{n}T_{i_{2}jl}^{+}L_{j}^{k}L_{l}^{k}\big(\hat{e}_{\lambda}(t_1,t_1-\xi_{i_2jl}(t_1))+
\hat{e}_{\lambda}(t_1,t_1-\zeta_{i_2jl}(t_1))\big)
\bigg)\\
&&+ c_{i_2}^{+}cP M\hat{e}_{\ominus_{\nu}\lambda}(t_{1},t_{0})\|\psi\|_{\mathbb{B}}\int_{t_{0}}^{t_{1}}\hat{e}_{-c_{i_2}}(t_{1},
\rho(s))\hat{e}_{\lambda}(t_{1},\rho(s))\\
&&\times\bigg\{
c_{i_2}^{+}\int_{s-\eta_{i_2}(s)}^{s}\hat{e}_{\lambda}(\rho(s),u) \nabla u
+\sum_{j=1}^{n}a_{i_2j}^{+}L_{j}^{f}\hat{e}_{\lambda}(\rho(s),s)\\
&&+\sum_{j=1}^{n}b_{i_2j}^{+}L_{j}^{g}\hat{e}_{\lambda}(\rho(s),s-\tau_{i_2j}(s))
+\sum_{j=1}^{n}d_{i_{2}j}^{+}L_{j}^{h}\int_{s-\sigma_{i_{2}j}(s)}^{s}\hat{e}_{\lambda}(\rho(u),u)\nabla u\\
&&+\sum_{j=1}^{n}\sum_{l=1}^{n}T_{i_{2}jl}^{+}L_{j}^{k}L_{l}^{k}\big(\hat{e}_{\lambda}(s,t_1-\xi_{i_2jl}(s))+
\hat{e}_{\lambda}(s,t_1-\zeta_{i_2jl}(s))\big)\bigg\}
\nabla s\\
&\leq&c_{i_2}^{+} e_{-c_{i_2}}(t_{1},t_{0})\|\psi\|_{\mathbb{B}}+cP
M\hat{e}_{\ominus_{\nu}\lambda}(t_{1},t_{0})\|\psi\|_{\mathbb{B}}
\bigg(c_{i_2}^{+}\delta_{i_2}^{+}\exp(\lambda\delta_{i_2}^{+})
+\sum_{j=1}^{n}a_{i_2j}^{+}L_{j}^{f}\\
&&+\sum_{j=1}^{n}b_{i_2j}^{+}L_{j}^{g}\exp(\lambda\tau_{i_2j}^{+})
+\sum_{j=1}^{n}d_{i_2j}^{+}L_{j}^{h}\sigma_{i_{2}j}^{+}\exp(\lambda\sigma_{i_2j}^{+})\\
&&+\sum_{j=1}^{n}\sum_{l=1}^{n}T_{i_{2}jl}^{+}L_{j}^{k}L_{l}^{k}\big(\exp(\lambda\xi_{i_2jl}^{+})
+\exp(\lambda\zeta_{i_2jl}^{+})\big)\bigg)\\
&&\times\bigg(1+c_{i_2}^{+}\exp(\lambda\sup\limits_{s\in\mathbb{T}}\nu(s))
\int_{t_{0}}^{t_{1}}\hat{e}_{-c_{i_2}\oplus\lambda}(t_{1},\rho(s))
\Delta s\bigg)\\
&\leq& cP M\hat{e}_{\ominus_{\nu}\lambda}(t_{1},t_{0})\|\psi\|_{\mathbb{B}}
\bigg\{\frac{c_{i_2}^{+}}{M}\hat{e}_{-(c_{i_2}^{-}-\lambda)}(t_{1},t_{0})
+\bigg(c_{i_2}^{+}\delta_{i_2}^{+}\exp(\lambda\delta_{i_2}^{+})
+\sum_{j=1}^{n}a_{i_2j}^{+}L_{j}^{f}\\
&&+\sum_{j=1}^{n}b_{i_2j}^{+}L_{j}^{g}\exp(\lambda\tau_{i_2j}^{+})
+\sum_{j=1}^{n}d_{i_2j}^{+}L_{j}^{h}\sigma_{i_{2}j}^{+}\exp(\lambda\sigma_{i_2j}^{+})\\
&&+\sum_{j=1}^{n}\sum_{l=1}^{n}T_{i_{2}jl}^{+}L_{j}^{k}L_{l}^{k}\big(\exp(\lambda\xi_{i_2jl}^{+})
+\exp(\lambda\zeta_{i_2jl}^{+})\big)\bigg)\\
&&\times\bigg(1+c_{i_2}^{+}\exp(\lambda\sup\limits_{s\in\mathbb{T}}\nu(s))\frac{1}{-(c_{i_2}^{-}-\lambda)}
\big(\hat{e}_{-(c_{i_2}^{-}-\lambda)}(t_{1},t_{0})-1\big)\bigg) \bigg\}\\
&\leq& cP M\hat{e}_{\ominus_{\nu}\lambda}(t_{1},t_{0})\|\psi\|_{\mathbb{B}}
\bigg\{\bigg[\frac{1}{M}-\frac{\exp(\lambda\sup\limits_{s\in\mathbb{T}}\nu(s))}{c_{i_2}^{-}-\lambda}
\bigg(c_{i_2}^{+}\delta_{i_2}^{+}\exp(\lambda\delta_{i_2}^{+})
\sum_{j=1}^{n}a_{i_2j}^{+}L_{j}^{f}\\
&&+\sum_{j=1}^{n}b_{i_2j}^{+}L_{j}^{g}\exp(\lambda\tau_{i_2j}^{+})
+\sum_{j=1}^{n}d_{i_2j}^{+}L_{j}^{h}\sigma_{i_{2}j}^{+}\exp(\lambda\sigma_{i_2j}^{+})\\
&&+\sum_{j=1}^{n}\sum_{l=1}^{n}T_{i_{2}jl}^{+}L_{j}^{k}L_{l}^{k}\big(\exp(\lambda\xi_{i_2jl}^{+})
+\exp(\lambda\zeta_{i_2jl}^{+})\big)\bigg)\bigg]c_{i_2}^{+}\hat{e}_{-(c_{i_2}^{-}-\lambda)}(t_{1},t_{0})\\
&&+\bigg(1+\frac{c_{i_2}^{+}\exp(\lambda\sup\limits_{s\in\mathbb{T}}\mu(s))}{c_{i_2}^{-}-\lambda}\bigg) \bigg(c_{i_2}^{+}\delta_{i_2}^{+}\exp(\lambda\delta_{i_2}^{+})\\
&&+\sum_{j=1}^{n}a_{i_2j}^{+}L_{j}^{f}+\sum_{j=1}^{n}b_{i_2j}^{+}L_{j}^{g}\exp(\lambda\tau_{i_2j}^{+})
+\sum_{j=1}^{n}d_{i_2j}^{+}L_{j}^{h}\sigma_{i_{2}j}^{+}\exp(\lambda\sigma_{i_2j}^{+})\\
&&+\sum_{j=1}^{n}\sum_{l=1}^{n}T_{i_{2}jl}^{+}L_{j}^{k}L_{l}^{k}\big(\exp(\lambda\xi_{i_2jl}^{+})
+\exp(\lambda\zeta_{i_2jl}^{+})\big)\bigg)\bigg\}\\
&<&cP M\hat{e}_{\ominus_{\nu}\lambda}(t_{1},t_{0})\|\psi\|_{\mathbb{B}}.
\end{eqnarray*}
The above two inequalities imply that
\begin{eqnarray*}
\|z(t_{1})\|_{\mathbb{B}}<cPM \hat{e}_{\ominus_{\nu}\lambda}(t_{1},t_{0})\|\psi\|_{\mathbb{B}},
\end{eqnarray*}
which contradicts \eqref{e48}, and so \eqref{e46} holds. Letting $P\rightarrow
1$, then \eqref{e45} holds. Hence, the pseudo almost periodic solution of system
\eqref{a1} is globally exponentially stable. The proof is complete.
\end{proof}

\section{ An example}
 \setcounter{equation}{0}

 \indent

In this section, we will give  an example to illustrate the feasibility
and effectiveness of our results obtained in Sections 3 and 4.

\begin{example}
Let $n=2$. Consider the following neural network system on time scale $\mathbb{T}$:
\begin{eqnarray}\label{d1}
x_{i}^{\nabla}(t)&=&-c_{i}(t)x_{i}(t-\delta_{i}(t))+\sum_{j=1}^{2}a_{ij}(t)f_{j}(x_{j}(t))\nonumber\\
&&+\sum_{j=1}^{2}b_{ij}(t)g_{j}(x_{j}(t-\tau_{ij}(t)))
+\sum_{j=1}^{2}d_{ij}(t)\int_{t-\sigma_{ij}(t)}^{t}h_{j}(x_{j}^{\nabla}(s))\nabla s\nonumber\\
&&+\sum_{j=1}^{2}\sum_{l=1}^{2}T_{ijl}(t)k_{j}(x_{j}(t-\xi_{ijl}(t)))k_{l}(x_{l}(t-\zeta_{ijl}(t)))
+I_{i}(t),
\end{eqnarray}
where $t\in \mathbb{T}$, $i=1,2$ and the coefficients are as follows:
\[c_{1}(t)=0.95+0.05\sin t,\,\,\,\,c_{2}(t)=0.91+0.04\sin t,\,\,\,\,a_{11}(t)=0.05\cos t
\]
\[a_{12}(t)=0.07\cos \sqrt{2}t,\,\,\,\,a_{21}(t)=0.05\cos\bigg(\frac{1}{3} t\bigg),\,\,\,\,a_{22}(t)=0.03\cos\bigg(\frac{3}{4}t\bigg),
\]
\[b_{11}(t)=0.06\sin t,\,\,\,\,b_{12}(t)=0.03\cos \sqrt{2}t,\,\,\,\,b_{21}(t)=0.04\cos t,\]
\[
b_{22}(t)=0.03\sin \sqrt{2}t,\,\,\,\,d_{11}(t)=0.08\sin t,\,\,\,\,d_{12}(t)=0.04\sin \sqrt{2}t,\,\,\,\,d_{21}(t)=0.06\sin t,
\]
\[d_{22}(t)=0.07\cos \sqrt{2}t,\,\,\,\,f_{1}(u)=0.1\sin u,\,\,\,\,f_{2}(u)=\sin u,\,\,\,\,g_{1}(u)=0.1\cos u,\,\,\,\,g_{2}(u)=\cos u,\]
\[h_{1}(u)=0.1\sin \frac{1}{3}u,\,\,\,\,h_{2}(u)=\sin \sqrt{u},\,\,\,\,k_{1}(u)=0.1\cos\sqrt{2 u},\,\,\,\,k_{2}(u)=\sin 3u,\]
\[T_{111}(t)=T_{222}(t)=0.075+0.025\sin \sqrt{2}t,\,\,T_{112}(t)=T_{212}(t)=0.07+0.03\cos\bigg(\frac{3}{4}t\bigg),\]
\[T_{121}(t)=T_{221}(t)=0.075+0.025\cos \sqrt{3}t,\,\, T_{122}(t)=T_{211}(t)=0.07+0.03\sin\bigg(\frac{3}{4}t\bigg).
\]
\[I_{1}(t)=I_{2}(t)=0.5\sin 2t,\,\,\,\,J_{1}(t)=J_{2}(t)=0.5\cos\sqrt{2}t.
\]
By a simple calculation, we have
\[c_{1}^{+}=0.1,\,\,\,c_{2}^{+}=0.95,\,\,\,c_{1}^{-}=0.9,\,\,\,c_{2}^{-}=0.0.87,
\]
\[a_{11}^{+}=0.05,\,\,\,a_{12}^{+}=0.07,\,\,\,\,a_{21}^{+}=0.05,\,\,\,a_{22}^{+}=0.03,
\]
\[b_{11}^{+}=0.06,\,\,\,b_{12}^{+}=0.03,\,\,\,b_{21}^{+}=0.04,\,\,\,b_{22}^{+}=0.03,
\]
\[d_{11}^{+}=0.08,\,\,\,d_{12}^{+}=0.04,\,\,\,d_{21}^{+}=0.06,\,\,\,d_{22}^{+}=0.07,
\]
\[T_{111}^{+}=T_{222}^{+}=T_{112}^{+}=T_{212}^{+}=T_{121}^{+}=T_{221}^{+}=T_{122}^{+}=T_{211}^{+}=0.01,
\]
\[H_{1}^{f}=H_{1}^{g}=H_{1}^{h}=H_{1}^{k}=0.1,\,\,\,H_{2}^{f}=H_{2}^{g}=H_{2}^{h}=H_{2}^{k}=1.\]
Therefore, whether $\mathbb{T}=\mathbb{R}$ or $\mathbb{T}=\mathbb{Z}$, all the conditions of Theorem \ref{thm31} and Theorem
\ref{thm41} are satisfied, hence, we know that system \eqref{d1} has a pseudo almost periodic solution, which is globally exponentially stable. This is, the continuous-time neural network and its discrete-time analogue have the same dynamical behaviors for the pseudo almost periodicity.

\end{example}

\end{document}